\documentclass{article}
\usepackage{amsmath,amssymb,amsthm,stmaryrd,makeidx}
\usepackage[all]{xy}

\DeclareMathOperator{\enm}{End}

\DeclareMathOperator{\im}{im}

\DeclareMathOperator{\hmm}{Hom}
\DeclareMathOperator{\mor}{Mor}

\DeclareMathOperator{\pts}{pts}

\newcommand{\cat}[1]{\mathbf{#1}}

\newcommand{\sets}{\mathbf{Sets}}

\newcommand{\grmcat}[1]

\input xypic

\newtheorem{proposition}{Proposition}

\newtheorem{lemma}{Lemma}

\newtheorem{definition}{Definition}

\makeindex

\begin{document}
\author{Arvid Siqveland}
\title{Localization in Associative Rings}

\maketitle

\begin{abstract} In \cite{S251} we constructed schemes of objects in small categories which contained a set of basepoints with local representing (localizing) objects. Here we prove that the category $\cat{Rings}$ of associative rings with unit has a certain set of basepoints for which localizing rings exist. We take the set of base points $B$ to be the set of rings on the form $\enm_{\mathbb Z}(M)$ where $M$ is a simple right $A$-module for some associative ring $A.$ The set of base-points in the associative ring $A$ is defined as $\pts_B(A)=\{\mor_{\cat{Rings}}(A,\enm_{\mathbb Z}(M))\}.$ For  any finite subset $M\subseteq\pts_B(A)$ we prove that the localizing ring $A_M$ exists.  and so the construction from \cite{S251} gives a definition of schemes of associative algebras. Defining a topology on $\pts_B(A)$ such that when $A$ is commutative it is the Zariski topology, we get the ordinary definition of schemes when we consider the category of commutative rings. This article is in line with the philosophy that a scheme is a moduli of its base-points.
\end{abstract}

\section{Introduction}

In constructing algebraic moduli varieties, the concept of localization is important. Defining moduli varieties by deformation theory, we observe that the defined local function rings are close to the usual localization in the commutative case, and so can be used to define local function rings in the noncommutative situation. These rings are not necessarily local, but we will never the less call them localizations, local representing rings in line with \cite{S251}, or we will call these rings localizing.

We constructed associative varieties over a field $k$ by deformation theory. This says that we observed that we can replace naturally the product  $\prod_{i=1}^r A_{\mathfrak m_i}$ of $r$ localizations in maximal ideals in the commutative situation, with a particular subring of a complete matrix algebra $\hat O^A_M=(H_{ij}\otimes_k\hmm_k(M_i,M_j))$ constructed from a direct sum of $r\geq 1$ simple $A$-modules $M=\oplus_{i=1}^rM_i.$ The matrix ring $H=(H_{ij})_{r\times r}$ is the noncomutative prorepresenting hull, also called the noncommutative local formal moduli, of the deformation functor of the set $M$ of $r\geq 1$ given modules. The associative construction should extend the commutative construction of varieties, at least in some particular cases. 

In the commutative situation, we can replace the prorepresening hull of one simple module by the completion of the $A_{\mathfrak m}$ where $\mathfrak m$ is the maximal ideal corresponding to the simple module $A/\mathfrak m,$ and this leads to the local ring $(A/\ker\kappa)_{\mathfrak m}$ where $\kappa:A\rightarrow \hat A_{\mathfrak m}$ is the natural morphism.

Our goal in this text is to prove that we can give a definition of localization of associative rings, based only on the theory of associative rings. Thus we need both localization (which is not necessarily local) and completion in the associative case. Deformation theory can then be applied for actual computation of the localizing rings, which is of course essential for obtaining result on the algebraic objects under study, i.e. their moduli spaces, see \cite{S242}.

This leads to a representation theoretic definition of local function rings, giving a way of making more structure to differential geometry, going further than in the commutative situation.

\section{Commutative Localization}
We assume that the reader is familiar with the definition of localization of a commutative ring in a prime ideal as described e.g. in \cite{AM69}. We give an equivalent definition based on categories and functors.

Let $A$ be a commutative ring and $\mathfrak p\subset A$ a prime ideal.
Let $\mathfrak m_B\subset B$ be a local ring and $f:A\rightarrow B$  a homomorphism such that $f^{-1}(\mathfrak m_B)=\mathfrak p.$
For a local homomorphism $\phi:\mathfrak m_B\subset B\rightarrow\mathfrak m_C\subset C$ we see that $f^{-1}(\phi^{-1}(\mathfrak m_C))=\mathfrak p.$
This says that we have the covariant functor $L_{\mathfrak p\subset A}:\cat{LRings}\rightarrow\cat{Sets}$ from the category of local rings to the category of sets given by $$L_{\mathfrak p\subset A}(\mathfrak m_B\subset B)=\{f:A\rightarrow B|f^{-1}(\mathfrak m_B)=\mathfrak p\}.$$ 

\begin{lemma}\label{comloclemma} The functor $$L_{\mathfrak p\subset A}:\cat{LRings}\rightarrow\sets$$ is represented by $A_{\mathfrak p}.$
\end{lemma}

\begin{proof} Let $h:A\rightarrow B$ be a homomorphism from $A$ into a local ring $\mathfrak m_B\subset B$ such that $h^{-1}(\mathfrak m_B)=\mathfrak p.$ Then $a\notin\mathfrak p$ implies that $h(a)$ is a unit, and so by the universal property of localization, there exists a unique homomorphism $\xi:A_\mathfrak p\rightarrow B$ commuting in the diagram $$\xymatrix{A\ar[r]^h\ar[dr]_{f_\mathfrak p}&B\\&A_\mathfrak p.\ar[u]_\xi}$$ We have to prove that $\xi$ is local. This follows as $\xi(f_\mathfrak p(\mathfrak p))\subseteq\mathfrak m_B$ implies that $f_{\mathfrak p}(\mathfrak p)\subseteq\xi^{-1}(\mathfrak m_B)$ so that $\xi^{-1}(\mathfrak m_B)=\mathfrak p A_{\mathfrak p}$ because $A_{\mathfrak p}$ is local with maximal ideal $\mathfrak p A_{\mathfrak p}.$
\end{proof}

We define the category of pointed rings $\cat{CRings}^\ast$ as the category with objects the pairs $\mathfrak p\subset A$ where $A$ is a commutative ring and $\mathfrak p$ is a prime ideal in $A.$ A morphism $\phi:\mathfrak p\subset A\rightarrow \mathfrak q\subset B$ is a homomorphism $\phi:A\rightarrow B$ such that $\phi^{-1}(\mathfrak q)=\mathfrak p.$ Then the proof of Lemma \ref{comloclemma} holds for the following.

\begin{proposition} The functor $\tilde L_{\mathfrak p\subset A}:\cat{CRings}^\ast\rightarrow\cat{Sets}$ defined by $$\tilde L_{\mathfrak p\subset A}(\mathfrak q\subset B)=\mor_{\cat{CRings}^\ast}(\mathfrak p\subset A,\mathfrak q B_{\mathfrak q}\subset B_{\mathfrak q})$$ is represented by $\mathfrak p A_{\mathfrak p}\subset A_{\mathfrak p}.$
\end{proposition}

\section{Hausdorff Localization}
Let $A$ be a commutative ring and $\mathfrak p\subset A$ a prime ideal.

\begin{definition} Let $\kappa:A\rightarrow \hat A_{\mathfrak p}$ be the canonical homomorphism. The Hausdorff  localization of $A$ in $\mathfrak p$ is defined as the local ring $H^A_{\mathfrak p}=(A/\ker\kappa)_{\mathfrak p}.$ 
\end{definition}

By exactness properties of localization, we have that the natural homomorphism $\rho:H^A_{\mathfrak p}\rightarrow \hat A_{\mathfrak p}$ is an injective local homomorphism.

Let Let $\mathfrak q\subset B$ be a prime ideal in ring $B$ and let $f:A\rightarrow \hat B_{\mathfrak q}$ be a homomorphism such that $\mathfrak p=f^{-1}(\hat{\mathfrak m}_{\mathfrak q}).$  Let $\hat f:\hat A_{\mathfrak p}\rightarrow\hat B_{\mathfrak q}$ be the natural homomorphism. We then have the commutative diagram

\begin{equation}\label{kereq}\xymatrix{A\ar[r]^f\ar[dr]_{\kappa}&\hat B_{\mathfrak q}\\&\hat A_{\mathfrak p}\ar[u]_{\hat f}}\end{equation} which proves that $\ker\kappa\subseteq\ker f.$

Consider the covariant functor $$\hat L_{\mathfrak p\subset A}:\cat{LRings}\rightarrow\cat{Sets}$$ given by $\hat L_{\mathfrak p\subset A}(\mathfrak m_B\subset B)=\{f:A\rightarrow \hat B|f^{-1}(\hat{\mathfrak m}_B)=\mathfrak p\}.$

\begin{lemma} The functor $\hat L_{\mathfrak p\subset A}$ is represented by $H^A_{\mathfrak p}.$
\end{lemma}

\begin{proof} The composition $A\rightarrow H^A_{\mathfrak p}\hookrightarrow\hat A_{\mathfrak p}$ gives a homomorphism $\eta:A\rightarrow \hat A_{\mathfrak p}$ such that $\eta^{-1}(\hat{\mathfrak m}_{\mathfrak p})=\mathfrak p.$ Given any homomorphism $f:A\rightarrow \hat B$ such that $f^{-1}(\hat m_B)=\mathfrak p$ where $B$ is a local ring with maximal ideal $\mathfrak m_B,$ we have the commutative diagram
$$\xymatrix{A\ar[r]^f\ar[d]&\hat B_{\mathfrak q}\\A/\ker\kappa\ar[ur]\ar[r]&(A/\ker\kappa)_{\mathfrak p}\ar[u]_{\exists!\xi}}$$ where $f$ factors through $A/\ker\kappa$ as proved in diagram $\ref{kereq}.$ The existence of the unique homomorphism $\xi$ follows from the universal property of localization, and as in the proof of Lemma \ref{comloclemma}, $\xi$ is a local homomorphism.
\end{proof}

This proof also directly proves the following:

\begin{proposition} The functor $\tilde{\hat L}_{\mathfrak p\subset A}:\cat{CRings}^\ast\rightarrow\cat{Sets}$ defined by $$\tilde{\hat L}_{\mathfrak p\subset A}(\mathfrak q\subset B)=\{f:A\rightarrow \hat B_{\mathfrak q}|f^{-1}(\hat{\mathfrak m}_{\mathfrak q})=\mathfrak p\}$$is represented by $H^A_{\mathfrak p}.$
\end{proposition}

\section{Associative Localizing Rings}
We follow the structure from the commutative situation, starting with a generalization of the ring of fractions in the complement of a maximal ideal.

Let $A$ be an associative ring and $M,N,P$ right $A$-modules. The study of associative rings and modules has strong connections to linear algebra.  
Given $A$-module homomorphisms $M\overset f\rightarrow N\overset g\rightarrow P$ we compose from left to right, as in matrix multiplication, such that the composition is $f\circ g:M\rightarrow P$ given by $f\circ g(m)=g(f(m))\in P.$ Then it is most convenient to consider right $A$-modules, where the right $A$-module structure on the abelian group $M$ is defined by the ring homomorphism, called the structure morphism, $\eta^A_M:A\rightarrow\enm_{\mathbb Z}(M),$ as  $ma=\eta^A_M(a)(m).$ Notice in particular that associativity holds because $$\begin{aligned}(ma)b&=\eta^A_M(b)(ma)=\eta^A_M(b)(\eta^A_M(a)(m))=\eta^A_M(a)\circ\eta^A_M(b)(m)\\&=\eta^A_M(ab)(m)=m(ab).\end{aligned}$$

A right $A$-module $M$ is simple if it contains no other proper submodules than the zero module.

\begin{lemma}[Shur's Lemma] When $M$ is a simple nonzero $A$-module, $\enm_A(M)$ is a division ring.
\end{lemma}

\begin{proof} Assume $f:M\rightarrow M$ is nonzero. Then as $\ker f\subseteq M$ is a submodule, $\ker f=0.$ Also, $\im f\subseteq M$ is a submodule, and so is $\im f=M.$ Thus $f$ is both surjective and injective, so an isomorphism. This says that $\enm_A(M)$ is a division ring.
\end{proof}

Notice that if  $a\in A$ is a unit, which means that there exists a left inverse $u_l$ with $u_la=1$ and a right inverse $u_r$ such that $au_r=1$ then $$u_l=u_l\cdot 1=u_l(au_r)=(u_la)u_r=1\cdot u_r=u_r.$$ 

Let $A$ be an associative ring and $$\eta^A_M:A\rightarrow\enm_{\mathbb Z}(M)$$ a simple $A$-module. As any $A$-linear homomorphism is also $\mathbb Z$-linear, we have a canonical ring homomorphism $$D_M=\enm_A(M)\rightarrow\enm_{\mathbb Z}(M)=E_M.$$ This homomorphism is nonzero (because its maps the identity to the identity), and so injective because $D_M$ is a division ring.

\begin{definition} The local function ring of $A$ in $M$ is the subring $A_M\subseteq E_M$ of $E_M$ generated over $\im\eta^A_M$ by the subset $\{\eta^A_M(s)^{-1}|\eta^A_M(s)\in D_M\setminus(0)\subset E_M\}.$
\end{definition}

\begin{lemma} The local function ring $A_M$ of $A$ in $M$ satisfy the following universal property:
There exists a homomorphism $\eta:A\rightarrow A_M$ such that $\eta(s)$ is a unit whenever $s\in\eta^{-1}(D_M\setminus (0)).$ If $B$ is any other associative algebra with a homomorphism $\kappa:A\rightarrow B$ such that $\kappa(s)$ is a unit whenever $s\in\eta^{-1}(D_M\setminus (0))$, and $\ker\eta\subseteq\ker\kappa,$ then there exists a unique homomorphism $\rho:A_M\rightarrow B$ such that $\eta\circ\rho=\kappa.$
\end{lemma}

\begin{proof} Assume given such a $\kappa:A\rightarrow B.$  If there is a homomorphism $\rho:A_M\rightarrow B$ such that $\eta\circ\rho=\kappa,$ we necessarily have $$1=\rho(\eta(s)\eta(s)^{-1})=\rho(\eta(s))\rho(\eta(s)^{-1})\Rightarrow\rho(\eta(s)^{-1})=\rho(\eta(s))^{-1}=\kappa(s)^{-1},$$ so that  $\rho(\eta(a)\eta(s)^{-1})=\kappa(a)\kappa(s)^{-1}.$ Thus $\rho$ is unique if it is exists. Now, the condition  $\ker\eta\subset\ker\kappa$ ensures that $\rho$ defined by $\rho(\eta(a)\eta(s)^{-1})=\kappa(a)\kappa(s)^{-1}$ is well defined.  
\end{proof}

\begin{lemma} Assume that $A$ is commutative and that $M\cong A/\mathfrak m$ where $\mathfrak m\subset A$ is a maximal ideal. Then  $$A_M=A_{A/\mathfrak m}\simeq A_{\mathfrak m}.$$
\end{lemma} 

\begin{proof} The ring $A_{\mathfrak m}$ satisfy the universal property of $A_M$ and so the homomorphism $A_M\rightarrow A_{\mathfrak m}$ is an isomorphism.
\end{proof}

\begin{definition} A Local Function Ring (LFR) is an associative ring $A_M$ with a simple module $\eta^A_M:A_M\hookrightarrow\enm_{\mathbb Z}(M)$ such that $\eta(s)$ is unit whenever $s\in\eta^{-1}(D_M\setminus(0)).$ There is a  morphism from one LFR $A_M$ to another $B_N$ if the abelian group $M=N,$ and then a morphism is a ring homomorphism $\phi:A_M\rightarrow B_M$ commuting in the diagram $$\xymatrix{A_M\ar[r]^\phi\ar[dr]_{\eta^A_M}&B_M\ar[d]^{\eta^B_M}\\&\enm_{\mathbb Z}(M).}$$
\end{definition}

\begin{proposition} A commutative LFR is a  localization in a maximal ideal, and a morphism is a local homomorphism.
\end{proposition}

\begin{proof} When $A$ is commutative, the simple modules are on the form $M=A/\mathfrak m$ with $\mathfrak m$ a maximal ideal. Then $A_M$ is contained in $\enm_{\mathbb Z}(M)$ and every $s\in A_M\setminus\mathfrak m$ is a unit. This says that $A_M\simeq A_{\mathfrak m}.$ 
\end{proof}

Let $A$ be an associative ring and let $\eta^A_M:A\rightarrow\enm_{\mathbb Z}(M)$ be a simple right $A$-module. Consider the covariant functor $L^A_M:\cat{LFR}\rightarrow\cat{Sets}$ given by $$L^A_M(B_N)=\{\phi:A\rightarrow B|\phi\circ\eta^B_N=\eta^A_M\}.$$ Notice that in general, it might happen that for $\phi:A\rightarrow B$ to exist, we need $M= N.$ That is to say that in general, it might happen that $L^A(B_N)=0.$ Also notice that the functor is covariant because given $\psi:B_N\rightarrow C_N$ we have $$\phi\circ\psi\circ\eta^C_N=\phi\circ\eta^B_N=\eta^A_M.$$

\begin{lemma} The functor $$L^A_M:\cat{LFR}\rightarrow\sets$$ is represented by $\eta:A\rightarrow A_M.$
\end{lemma}

\begin{proof} Let $\kappa:A\rightarrow B_M$ be an element of $L^A_M(B_M),$ that is, there is a commutative diagram $$\xymatrix{A\ar[dr]_{\eta^A_M}\ar[r]^\kappa&B_M\ar[d]^{\eta^B_M}\\&\enm_{\mathbb Z}(M)}$$ where $B$ is an LFR. Because $\enm_A(M)$ and $\enm_B(M)$ maps to the same set $D_M\subseteq\enm_{\mathbb Z}(M),$ we have that $\eta^A_M(s)\in D_M\setminus (0)\Rightarrow\eta^B_M(\kappa(s))\in D_M\setminus(0)$ so that $\kappa(s)$ is a unit when $s\in(\eta^A_M)^{-1}(D_M\setminus(0)).$  By definition of the category $\cat{LFR},$ the homomorphism $\eta^B_M$ is injective, so that $\ker\eta^A_M\subseteq\ker\kappa.$ By the universal property of $A_M,$ there is a unique morphism $\rho:A_M\rightarrow B_M$ such that $\eta\circ\rho=\kappa.$
\end{proof}

\begin{definition}
Let $M=\oplus_{i=1}^rM_i$ be a direct sum of $r$ simple right $A$-modules $\eta^A_i:A\rightarrow\enm_{\mathbb Z}(M_i),\ 1\leq i\leq r.$ We define the $r$-local function ring in $M$ as $$A_M=\underset{1\leq i\leq r}\prod\ A_{M_i}.$$ 
\end{definition}

In Section \ref{products} we will prove that the product above exists. Then we find that the compositions $\eta^{A_M}_{M_i},$ $A_M\rightarrow A_{M_i}\rightarrow\enm_{\mathbb Z}(M_i),$ are simple right $A_M$-modules.

\section{Hausdorff Localizing Associative Rings}

Let $A$ be an associative ring and $\eta^A_M:A\rightarrow\enm_{\mathbb Z}(M)$ a simple right $A$-module. Then $$\eta^{A_M}_M:A_M\rightarrow\enm_{\mathbb Z}(M)$$ is a simple $A_M$-module, commuting in the diagram 
$$\xymatrix{A\ar[r]^{\kappa}\ar[dr]_{\eta^A_M}&\hat A_M\ar[d]^{\eta^{\hat A_M}_M}\\&\enm_{\mathbb Z}(M)}$$

with $\mathfrak m=\ker\eta^{A_M}_M$ and $\hat{A}_M$ the $\mathfrak m$-adic completion  of $A_M.$

\begin{definition}\label{Hdef} Let $\kappa:A\rightarrow \hat A_M$ be the canonical homomorphism. The Hausdorff  localization of $A$ in $M$ is defined as the local function ring $H^A_M=(A/\ker\kappa)_M.$ 
\end{definition}

When $\phi:A\hookrightarrow B$ is an injective homomorphism of associative rings and  $\eta^A_M:A\rightarrow\enm_{\mathbb Z}(M)$ is a simple $B$-module as simple $A$-module, that is commuting in the diagram $$\xymatrix{A\ar[dr]_{\eta^A_M}\ar[r]^\phi&B\ar[d]^{\eta^B_M}\\&\enm_{\mathbb Z}(M),}$$
we cannot prove that the natural homomorphism $\phi_M:A_M\rightarrow B_M$ is injective in general. Thus localization is not necessarily right exact for associative rings, and so in general, $H^A_M$ is not embedded into $\hat A_M.$

Consider the covariant functor $$\hat L^A_M:\cat{LFR}\rightarrow\sets$$ given by $\hat L^A_M(B_N)=\{\phi:A\rightarrow \hat B_M|\phi\circ\eta^{\hat B_M}_M=\eta^A_M\}.$

\begin{lemma} The functor $\hat L^A_M$ is represented by $H^A_M.$
\end{lemma}

\begin{proof} Let $\phi:A\rightarrow \hat B_M$ be a morphism commuting in the diagram $$\xymatrix{A\ar[r]^\phi\ar[dr]_{\eta^A_M}&\hat B_M\ar[d]^{\eta^{\hat B_M}_M}\\&\enm_{\mathbb Z}(M).}$$ By the universal property of $B_M$ inducing the same property on $\hat B_M,$ $\phi(a)$ is a unit in $\hat B_M$ whenever $a\in(\eta^A_M)^{-1}(D_M\setminus\{0\})$ and so there is a unique morphism $\rho:A_M\rightarrow\hat B_M$ such that $\eta^A_M\circ\rho=\phi.$ This composes to a unique morphism $\bar\rho:(A/\ker\kappa)_M=H^A_M\rightarrow\hat B_M$ such that  $\eta^A_M\circ\bar\rho=\phi.$
\end{proof}

Let $\cat{ARings}^\ast$ be the category of pairs $(A,\eta^A_M)$ with $A$ associative and $M$ a simple right $A$-module, and morphisms $\phi:(A,\eta^A_M)\rightarrow (B,\eta^B_M)$ such that $\phi\circ\eta^B_M=\eta^A_M.$

The above proof directly proves the following:

\begin{proposition} The functor $\tilde{\hat L}^A_M:\cat{ARings}^\ast\rightarrow\sets$ defined by $$\tilde{\hat L}^A_M(\eta^B_M)=\{\phi:A\rightarrow \hat B_M|\phi\circ\eta^{\hat B_M}_M=\eta^A_M\}.$$ is represented by $H^A_M.$
\end{proposition}

\section{Products in Associative Rings}\label{products}

In the commutative situation, there exists a categorical product of the local rings. This is used to define schemes of commutative rings. We will prove that the direct product of localizing rings exists also in the category of associative rings, and then the main result in \cite{S251} proves the existence of schemes of associative rings.

Let $A$ be an associative ring and let $M=\oplus_{i=1}^rM_i$ be a direct sum of $r$ right simple $A$-modules.  We have a canonical ring homomorphism 
$$\iota:D_M=\oplus_{i=1}^r\enm_A(M_i)\rightarrow(\hmm_{\mathbb Z}(M_i,M_j))=E_M$$ where $D_M=\oplus_{i=1}^r\enm_A(M_i)$ is interpreted as the corresponding $r\times r$ diagonal matrix and $(\hmm_{\mathbb Z}(M_i,M_j))$ is an $r\times r$ matrix.
 This homomorphism is nonzero (because its maps the identity to the identity), and so injective on its components because each $D_{M_i}=\enm_A(M_i)$ is a division ring. We denote by $D_M^\ast=\{s\in D_M|\iota_i(s)\neq 0\text{ for all } 1\leq i\leq r\}$ which is the set of units (invertible elements) in $D_M.$ Let $$\eta^A_M:A\rightarrow \enm_\mathbb Z(M)=(\hmm_\mathbb Z(M_i,M_j))$$ be the structure morphism of $M=\oplus_{i=1}^rM_i.$

\begin{definition} The local function ring of $A$ in $M$ is the subring $A_M\subseteq E_M$ of $E_M$ generated over $\im\eta^A_M$ by the subset $\{\eta^A_M(s)^{-1}|\eta^A_M(s)\in D_M^\ast\subset E_M\}.$
\end{definition}

\begin{lemma} The local function ring $A_M$ of $A$ in $M$ satisfy the following universal property:
There exists a homomorphism $\eta:A\rightarrow A_M$ such that $\eta(s)$ is a unit whenever $s\in\eta^{-1}(D_M^\ast).$ If $B$ is any other associative algebra with a homomorphism $\kappa:A\rightarrow B$ such that $\kappa(s)$ is a unit whenever $s\in\eta^{-1}(D_M^\ast)$, and $\ker\eta\subseteq\ker\kappa,$ then there exists a unique homomorphism $\rho:A_M\rightarrow B$ such that $\eta\circ\rho=\kappa.$
\end{lemma}

\begin{proof} Assume given such a $\kappa:A\rightarrow B.$  If there is a homomorphism $\rho:A_M\rightarrow B$ such that $\eta\circ\rho=\kappa,$ we necessarily have $$1=\rho(\eta(s)\eta(s)^{-1})=\rho(\eta(s))\rho(\eta(s)^{-1})\Rightarrow\rho(\eta(s)^{-1})=\rho(\eta(s))^{-1}=\kappa(s)^{-1},$$ so that  $\rho(\eta(a)\eta(s)^{-1})=\kappa(a)\kappa(s)^{-1}.$ Thus $\rho$ is unique if it is exists. Now, the condition  $\ker\eta\subset\ker\kappa$ ensures that $\rho$ defined by $\rho(\eta(a)\eta(s)^{-1})=\kappa(a)\kappa(s)^{-1}$ is well defined.  
\end{proof}

\begin{lemma}\label{comlemma1} Assume that $A$ is commutative and that $M_i\cong A/\mathfrak m_i$ where $\mathfrak m_i\subset A$ is a maximal ideal for $1\leq i\leq r$. Then  $$A_M\simeq\prod_{i=1}^rA_{\mathfrak m_i}.$$
\end{lemma} 

\begin{proof} The ring $\prod_{i=1}^rA_{\mathfrak m_i}$ satisfy the universal property of $A_M$ and so the homomorphism $A_M\rightarrow \prod_{i=1}^rA_{\mathfrak m_i}$ is an isomorphism.
\end{proof}

Because of Lemma \ref{comlemma1} it is reasonable to call $A_M$ the (categorical) product of the local function rings, i.e. $A_M=\prod_{i=1}^r A_{M_i}$.

\begin{definition} An $r$-pointed Local Function Ring (LFR) is an associative ring $A_M$ with $r$ simple modules $\eta^A_M:A_M\hookrightarrow\hmm_{\mathbb Z}(M_i,M_j)$ such that $\eta(s)$ is unit whenever $s\in\eta^{-1}(D_M^\ast).$ There is a  morphism from one LFR $A_M$ to another $B_N$ if the abelian groups $M_i=N_i,$ and then a morphism is a ring homomorphism $\phi:A_M\rightarrow B_M$ commuting in the diagram $$\xymatrix{A_M\ar[r]^\phi\ar[dr]_{\eta^A_M}&B_M\ar[d]^{\eta^B_M}\\&\enm_{\mathbb Z}(M).}$$
\end{definition}

From the above definition, working with $\ker\eta^A_M$ in the case where $M=\oplus_{i=1}^r M_i,$ we put $\mathfrak m=\ker\eta^A_M$ and define Hausdorff localization by taking projective limits over $A_M/\mathfrak m^n$ and we obtain the corresponding results as in the case with only one module: That is, let
$\hat L^A_M:\cat{LFR}(r)\rightarrow\sets$ be given by $$\hat L^A_M(B_N)=\{\phi:A\rightarrow \hat B_M|\phi\circ\eta^{\hat B_M}_M=\eta^A_M\}.$$

Now, see Definition \ref{Hdef} and define $H^A_M$ correspondingly. The following proof follows word by word.

\begin{lemma} The functor $\hat L^A_M$ is represented by $H^A_M.$
\end{lemma}

\end{document}